\documentclass[12pt,a4paper]{amsart}
\usepackage{amsfonts}
\usepackage{amsthm}
\usepackage{amsmath}
\usepackage{amssymb}
\usepackage{amscd}
\usepackage[latin2]{inputenc}
\usepackage{t1enc}
\usepackage[mathscr]{eucal}
\usepackage{indentfirst}
\usepackage{graphicx}
\usepackage{graphics}
\usepackage{pict2e}
\usepackage{epic}
\numberwithin{equation}{section}
\usepackage[margin=2.9cm]{geometry}
\usepackage{epstopdf} 
\usepackage{verbatim}
\usepackage{color}
\usepackage{enumitem}

\newtheorem{thm}{Theorem}[section]
\theoremstyle{plain}
\newtheorem*{thm*}{Theorem}
\newtheorem{prop}[thm]{Proposition}
\newtheorem*{prop*}{Proposition}
\newtheorem{lem}[thm]{Lemma}
\newtheorem{cor}[thm]{Corollary}
\theoremstyle{definition}

\newtheorem*{ex*}{Example}
\theoremstyle{definition}

\newtheorem{remark}[thm]{Remark}

\DeclareMathOperator{\iind}{Ind}
\DeclareMathOperator{\SL}{SL}
\DeclareMathOperator{\GL}{GL}

\DeclareMathOperator{\supp}{supp}

\DeclareMathOperator{\Prim}{Prim}
\newcommand{\ca}[1]{\mathcal{#1}}
\newcommand{\fH}[1]{\mathcal{H}_{#1}}

\newcommand{\ind}[3]{\iind_{#1}^{#2}{#3}}

\newcommand{\s}[0]{\Sigma}
\newcommand{\gus}[0]{G^{(0)}}

\newcommand{\gudu}[0]{G^{u}_{u}}
\newcommand{\gud}[2]{G^{#1}_{#2}}

\newcommand{\op}{\pi_{\omega}}

\begin{document}
	\title{THE ORBIT SPACES OF GROUPOIDS WHOSE $C^*$-ALGEBRAS ARE CCR}
			
	\author[DW v Wyk]{DANIEL W. VAN WYK}
	
	\address{Federal University of Santa Catarina \\ Department of Mathematics \\
		Campus Universitário Trindade CEP 88.040-900, Florianópolis-SC, Brazil} 
	
	\email{dwvanwyk79@gmail.com}

	\subjclass[2010]{``20, 37, 46, 47.''}

	\keywords{Groupoids, Operator Algebras, $C^*$-Algebras, Representations, Liminal, CCR}

	\begin{abstract}
		Let $G$ be second-countable locally compact Hausdorff groupoid
		with a continuous Haar system.  We show that if the groupoid $C^*$-algebra of $G$ is CCR then the orbits of $G$ are closed and the stabilizers of $G$ are CCR. In particular, we remove the assumption of amenability in a theorem of Clark.  
	\end{abstract}
	
	\maketitle

\section*{INTRODUCTION}
Let $\ca{K}(\fH{})$ denote the compact operators on a Hilbert space $\fH{}$.
A $C^*$-algebra $\ca{A}$ is called CCR, if for every irreducible representation
$\pi:\ca{A}\to B(\fH{\pi})$ we have $\pi(\ca{A})=\ca{K}(\fH{\pi})$, and 
 $\ca{A}$ is called GCR if, for every irreducible representation
$\pi$, we have $\pi(\ca{A})\supset\ca{K}(\fH{\pi})$.

In \cite[Theorem 6.1]{Cla07} Clark gives the following CCR characterization for 
$C^*$-algebras of groupoids: 

\begin{thm*}[Clark]
Let $G$ be a second-countable locally compact Hausdorff groupoid with a Haar system. 
Suppose that  all the stabilizers of $G$ are amenable. Then the full $C^*$-algebra $C^*(G)$ of $G$ is CCR if and only if the orbit space of $G$ is $T_{1}$ and the stabilizers of $G$ are CCR. 
\end{thm*} 

Clark's theorem  generalizes a CCR characterization 
for transformation group $C^*$-algebras of Williams \cite{Wil81} to groupoids. However, Gootman's analogous GCR characterization \cite{Goo73} for transformation group $C^*$-algebras does not assume amenable stabilizers. Williams deals with possibly non-separable transformations groups, and therefore assumes amenable stabilizers at points where the stabilizer map is discontinuous. The lack of 
an amenability assumption in Gootman's GCR characterization led Clark  to
conjecture that the amenability hypothesis in the groupoid case
is unnecessary \cite{Cla07}. In \cite{Wyk18} Clark's conjecture is show to be correct in the GCR case. In this paper we show that Clark's conjecture is also true in the CCR case. 

We point out that the techniques used in the GCR case (\cite{Wyk18}) are very different from the techniques we use in this paper. In \cite{Wyk18} we relied on the fact a $C^*$-algebra is GCR if and only if it is a type I $C^*$-algebra (\cite{Sak67}). We could therefore use von Neumann algebras and the theory of direct integrals by constructing a direct integral representation of the groupoid $C^*$-algebra and give a necessary condition for the direct integral representation to be type I. Every CCR $C^*$-algebra is GCR, but the converse does not hold in general. Hence the ``GCR - type I equivalence'' is not helpful in the CCR case.

Careful analysis of Clark's proofs shows the amenability hypothesis is used to show continuity of a function from the orbit space of the groupoid into the spectrum of the $C^*$-algebra of the groupoid (\cite[Lemma 5.4]{Cla07}). Continuity of this function is only used to show that if $C^*(G)$ is CCR, then (i) the orbit space is $T_1$ and (ii) the stabilizers are CCR. However, Clark uses the continuity of this function only indirectly to prove statement (ii), since statement (ii) is shown under the assumption that the orbits are closed (that is, that the orbit space is $T_1$). Hence, once we show that if $C^*(G)$ is CCR then the orbit space of $G$ is $T_1$, then (ii) will follow with exactly the same argument that Clark uses.

We obtained a broad outline for a CCR proof for transformations groups from Astrid an Huef, which she credited to Elliot Gootman. Before considering the groupoid case, we worked out the details in the transformation group case based on Gootman's outline and  Williams' proof for possibly non-separable transformation groups.   
However, even in the transformation group case we needed to make adjustments to make the arguments work. Specifically, we realized that we need to generalize to vector valued function spaces, because we consider induced representations other than those induced from trivial representations. Our CCR proof for groupoid $C^*$-algebras uses similar arguments to those that we used for transformation groups, which again required some work to generalize groupoids. To our knowledge no CCR proof has been published for second-countable transformation groups, that does not require amenability. Thus, our result is new even for transformation groups. 

The structure of this paper is as follows. In Section \ref{ch:Prelim} the necessary definitions and formulas needed are given. In Section \ref{sec:MultiOpsEq} our main goal is to prove Proposition \ref{mainEquivReps}, which shows the equivalence of induced representations of $C^*(G)$ to representations by multiplication operators of the $C^*$-algebra $C_0(\gus)$ of continuous functions that vanish at infinity on the unit space of $G$. To do this we work in the multiplier algebra of $C^*(G)$ and thus describe a *-homomorphism of $C_0(\gus)$ into this multiplier algebra.  Section \ref{sec:CCRmain} contains main result, Theorem \ref{mainCCRthm}, which crucially relies on Lemma \ref{lem:EqReps}, also presented in this section. Lemma \ref{lem:EqReps} will give us a desired contradiction to an equivalence of representations as multiplication operators, which is established in Theorem \ref{mainCCRthm}. Lastly, in Section \ref{chap:Examples} we give two examples. The first is an example of a groupoid with non-amenable stabilizers where it's associated $C^*$-algebra is GCR but not CCR. The second is an example of a groupoid with non-amenable stabilizers whose $C^*$-algebra is CCR.



\section{PRELIMINARIES} \label{ch:Prelim}
Throughout $G$ is a second-countable locally compact Hausdorff groupoid with unit space $\gus$, and a continuous Haar system $\{\lambda^u\}_{u\in\gus}$ (see \cite{Ren80} for these definitions). Let $r$ and $s$ denote the \textit{range} and \textit{source} maps, respectively, from $G$ onto $\gus$. Let $u\in\gus$. Then $G^u:=r^{-1}(u)$, $G_u:=s^{-1}(u)$ and the \textit{stabilizer} at $u$ is $\gudu:=r^{-1}(u)\cap s^{-1}(u)$. Note that each stabilizers is a group. 
For $x\in G$, the map $R(x):=(r(x),s(x))$ defines an 
equivalence relation $\sim$ on $\gus$. Let 
$[u]:=\{v\in\gus :u\sim v\}$ denote the \textit{orbit} of $u$. The \textit{orbit space} $\gus / G$
is the quotient space for this equivalence relation. 

Throughout $C_{c}(X)$ denotes the continuous compactly 
supported functions from the topological space $X$ into $\mathbb{C}$.
	If $f,g\in C_{c}(G)$, then
	$$f\ast g (x):=\int_{G}{f(y)g(y^{-1}x) \: d\lambda^{r(x)}(y)}$$
	and
	$$ f^{*}(x):= \overline{f(x^{-1})} ,$$
	define convolution and involution operations on $C_{c}(G)$, respectively.  
	With these operations $C_{c}(G)$ is a *-algebra. Let $\fH{}$ be a separable Hilbert space and $B(\fH{})$ the bounded linear operators on $\fH{}$.  A \textit{representation} of $C_{c}(G)$ is a *-homomorphism $\pi:C_{c}(G)\to B(\fH{})$	such that $||\pi(f)||\leq ||f||_I$, where $||f||_I$ is the I-norm on $C_{c}(G)$ (see \cite{Ren80} for the I-norm). Then $C^{*}(G)$ is the completion of $C_{c}(G)$ in the norm
$$||f||:=\{\sup||\pi(f)||:\pi \text{ is a representation of }C_{c}(G)\}.$$	
We assume all representations are non-degenerate.

Throughout we  induce representations which are irreducible. Following Ionescu and Williams' construction of induced representations \cite{IonWil09}, we induce from closed subgroupoids. Below we give the required details of what we need from \cite{IonWil09}. All our induced representations are from closed subgroups, which we view as subgroupoids. This view only matters when taking the integrated form of a representation of a subgroup, as the integrated form of a groupoid representation  has a modular function in the integrand whereas the group version does not (see (\ref{eq:IntergfromGeneral}) below). 

W now give the necessary formulas of induced representations that we will need. The reader is referred to \cite{IonWil09} for more details on induced representations. For the remainder of this section let $u\in \gudu$, let $ \phi,\psi\in C_{c}(G_u)$ and let $f\in C_{c}(G)$. Since $G_u$ is closed in $G$, every $\phi\in C_{c}(G_u)$ has an extension $f_\phi$ to $C_{c}(G)$. Then $C_{c}(G_u)$ is a right $C_c(\gudu)$-pre-Hilbert module with the inner product given by 
\begin{equation}\label{eq:ModuleInnerProd}
\langle \psi, \phi \rangle_{C_{c}(\gudu)}(t):=\int_{G}{\overline{\psi(\gamma)} 
	\:\phi(\gamma t)	\lambda_u(\gamma)}  = f_\psi^{\ast} \ast f_\phi, \nonumber				
\end{equation}
as long as we remember that the product $f_\psi^{\ast} \ast f_\phi$ is  restricted to $\gudu$.
For all $\gamma \in G_u$ the equation
\begin{equation*}\label{eq:AdjointActionOnModule}
f\cdot \phi (\gamma):=\int_{G}{f(\eta)\phi(\eta^{-1}\gamma)\: d\lambda^{r(\gamma)}(\eta)}=f\ast f_\phi
\end{equation*}
defines an action of $C_c(G)$ as adjointable operators on $C_{c}(G_u)$. Here the product $f\ast f_\phi$ is restricted to $G_u$. 
The completion $\overline{C_{c}(G_u)}$ of $C_{c}(G_u)$ gives a right  $\overline{C_{c}(G_u)}$-Hilbert module on which $C^*(G)$ acts a adjointable operators (\cite[Lemma 2.16]{MoritaEq}).

Suppose  $\omega:\gudu \to B(\fH{})$ is a representation 
on a Hilbert space $\fH{}$. Fix $a\in C_c(\gudu)$, $h,k\in\fH{}$ and let $\beta^u$ be a Haar measure on $\gudu$ with corresponding modular function $\Delta_u$. 
Let
$\op:C^*(\gudu)\to B(\fH{})$ denote the integrated from of $\omega$, that is,
\begin{equation} \label{eq:IntergfromGeneral}
 \pi_{\omega}(a)h:=\int_{\gudu} {a(r)\:\Delta(r)^{-1/2} \:\omega_rh\:d\beta^u(r)},   
\end{equation}
(see \cite[Proposition 1.7]{Ren80}).
Let $C_{c}(G_u)\odot \fH{}$ be the algebraic tensor product. Define an 
inner product on $C_{c}(G_u)\odot \fH{}$ by 
\begin{equation}\label{eq:InduedHilbertInner}
\left( \phi\otimes h \mid \psi\otimes k \right):=\left( \pi_{\omega}(\langle \psi,\phi\rangle_{_{C^*(\gudu)}})\: h \mid k \right),
\end{equation}
for all elementary tensors $\phi\otimes h,\psi\otimes k\in C_{c}(G_u)\odot \fH{}$.
Denote the completion of $C_{c}(G_u)\odot \fH{}$ with respect to this inner 
product by $\overline{C_{c}(G_u)\odot \fH{}}$. Define the induced representation 
$\ind{\gudu}{G}{\op}:C_{c}(G)\to B(\overline{C_{c}(G_u)\odot \fH{}})$ by 
$$ \ind{\gudu}{G}{\op}(f) (\phi\otimes h) := (f\cdot \phi)\otimes h.$$
By \cite[Proposition 2.66]{MoritaEq}, $ \ind{\gudu}{G}{\op}$ extends to 
give a representation of $C^*(G)$ as bounded linear 
operators on the Hilbert space $\overline{C_{c}(G_u)\odot \fH{}}$. 

Finally, any irreducible representation of a stabilizer induces to an irreducible representation of $C^*(G)$, \cite[Theorem 5]{IonWil09}.


\section{REPRESENTATIONS BY MULTIPLICATION OPERATORS EQUIVALENT TO INDUCED REPRESENTATIONS}  \label{sec:MultiOpsEq}
Let $\mathcal{M}(C^*(G))$ denote the multiplier algebra of $C^*(G)$, and let $C_0(\gus)$ denote the continuous functions on $\gus$ that vanish at infinity. For $u\in\gus$ and an irreducible representation $\pi$ of $C^*(\gudu)$ we let $L_u$ denote the extension of $\ind{\gudu}{G}{\pi}$ to $\mathcal{M}(C^*(G))$.
In this section we show that, for every $u\in\gus$, there is a representation $M_u$ of  $C_{0}(\gus)$ which acts as 
multiplication operators on an appropriate $L^2$-space. Moreover, there is a *-homomorphism $V:C_0(\gus)\to \mathcal{M}(C^*(G))$ such that $M_u$ is unitarily equivalent to $L_u\circ V$.

Fix $u\in \gus$. For $\gamma\in \gud{}{u}$, let 
$$\dot{\gamma}:=\gamma\gudu=\{\gamma t:t\in \gudu\}.$$ 
Then 
$$\gud{}{u}/\gudu:=\{\dot{\gamma}:\gamma\in\gud{}{u}\}$$
is a partition of $\gud{}{u}$ and we give $\gud{}{u}/\gudu$ the quotient topology. 

The following two lemmas are almost certainly well-known. We include their proofs for the reader's convenience.

\begin{lem}\label{OpenQuotientMap}
	Let $G$ be a locally compact Hausdorff groupoid. Fix $u\in\gus$.
	Then the quotient map $q:\gud{}{u} \to \gud{}{u}/\gudu$ is open. 
	
\end{lem}
	\begin{proof}
	 Suppose that $V\subset\gud{}{u}$ is open. We claim that, for every fixed $a\in\gudu$, \newline
	 $Va:=\{\gamma a:\gamma\in V\}$ is open. To see this, define 
	 $$\phi_a(\gamma):=\gamma a,$$ 
	 for every $\gamma\in G_u$. Then, since multiplication is continuous
	 in $G$, it follows that $\phi_a$ is continuous. It is not hard to see that 
	 $\phi_a$ is bijective, and that $\phi_{a^{-1}}(\gamma):=\gamma a^{-1}$
	 is a continuous inverse for $\phi_a$. Thus $\phi_a$ is a homeomorphism
	 of $G_u$, and $Va=\phi_a(V)$ is open as claimed. 
	 
	 Since $Va$ is open, we also have that
	 \begin{equation}
	   q^{-1}(q(V))=\bigcup_{a\in\gudu}Va. \label{eq:openUnion}
	\end{equation}	
	 is open. By the definition of a quotient topology $q(V)$ is open in 
	 $\gud{}{u}/\gudu$ if and only if $q^{-1}(q(V))$ is open in $\gud{}{u}$. Thus, 
	  Equation (\ref{eq:openUnion}) shows that 
	 $q(V)$ is open in $\gud{}{u}/\gudu$, completing the proof.
	\end{proof}

\begin{lem} \label{quotientHausdorff}
	Let $G$ be a second-countable locally compact Hausdorff groupoid. Fix $u\in\gus$.
	Then $\gud{}{u}/\gudu$ is second-countable locally compact  and Hausdorff in the
	quotient topology.
\end{lem}
\begin{proof}
Since $G_u$ is closed in $G$, it follows that $G_u$ is second-countable locally 
compact and Hausdorff. It is not hard to see that $G_u$ is a proper right $\gudu$-space with the groupoid operation for the action, and that the
orbit space of this $\gudu$-space is precisely $\gud{}{u}/\gudu$. Then   $\gud{}{u}/\gudu$  is second-countable and locally compact by \cite[Lemma 3.35]{CrossedProd}. 
Moreover, since the action is proper, it follows from \cite[Corollary 3.43]{CrossedProd} that  $G_u/\gudu$ is Hausdorff.
\end{proof}

In Lemma \ref{Bruhat} below we combine and restate \cite[Lemma 3.2 and Lemma 3.3]{Cla07} as they form an important part of this section.

\begin{lem} \label{Bruhat}
Fix $u\in\gus$. Let $\beta$ be a Haar measure on $\gudu$ with modular function 
$\Delta_{u}$. Then
	\begin{enumerate}
	\item[(i)] there is a a non-negative, continuous function $b:G_u\to\mathbb{R}$ such that
	for any compact subset $K\subset G_u $, the support of $b$ and 
	$K\gudu:=\{\gamma t:\gamma\in K,t\in\gudu\}$
	have compact intersection, and
	$$\int_{\gudu}{b(\gamma t)\: d\beta(t)}=1,  $$
	for all $\gamma\in G_u$; 
	\item[(ii)] the function  
	$$ Q(f)(\dot{\gamma}):=\int_{\gudu}{f(\gamma t)} \:d\beta(t) $$
	defines a surjection from $C_{c}(G)$ onto $C_{c}(G_u/\gudu)$; 
	\item[(iii)] there is a continuous 
	function $\rho_{u}:\gud{}{u} \to (0,\infty)$
	such that, for $\gamma\in \gud{}{u}$ and $t\in \gudu$,
	$$\rho_u(\gamma t)=\rho_u(\gamma)\:\Delta_{u}(t);$$ and
	\item[(iv)] there is a Radon measure $\sigma_{u}$ on $G_u/\gudu$ such that
	$$ \int_{G_u}{f(\gamma)\rho_u({\gamma}) \: d\lambda_u(\gamma)}
	=\int_{G_u /\gudu}{\int_{\gudu}{f(\gamma t) \:d\beta(t)}\: d\sigma_{u}(\dot{\gamma})}, $$
	for all $f\in C_{c}(G_u)$.
	\end{enumerate}
\end{lem}

For the remainder of this section fix $u\in\gus$, 
an irreducible representation $\omega$ of $\gudu$ acting on a Hilbert space
$\fH{}$ and let $\sigma_u$ be the measure on $G_u/\gudu$ given by Lemma \ref{Bruhat}(iv).  
Consider the set 
\begin{eqnarray*}
\mathcal{L}^2(G_u /\gudu,\fH{},\sigma_u):=
\{f:G_u /\gudu\to\fH{}: && f \text{ is measurable and}   \hspace{1cm} \\
&& \int_{G_u /\gudu}{||f(\dot{\gamma})||^2\:d\sigma_u(\dot{\gamma})<\infty\}}.
\end{eqnarray*} 
Let $L^2(G_u /\gudu,\fH{},\sigma_u)$ denote the quotient of $\mathcal{L}^2(G_u /\gudu,\fH{},\sigma_u)$ where functions agreeing almost everywhere are equivalent.
Then $L^2(G_u /\gudu,\fH{},\sigma_u)$ is a Hilbert space 
with inner product 
$$(f\mid g)_{L^2}:= \int_{G_u /\gudu}{(f(\dot{\gamma})\mid g(\dot{\gamma}))_{\fH{}}
\:d\sigma_{u}(\dot{\gamma})}.$$
As is common in the literature we will not distinguish between a class in \newline $L^2(G_u /\gudu,\fH{},\sigma_u)$ and a representative of the class in $\mathcal{L}^2(G_u /\gudu,\fH{},\sigma_u)$.
Define \newline 
$M_u :C_{0}(\gus)\to B(L^2(G_u /\gudu,\fH{},\sigma_u))$  by
\begin{equation} \label{DfnMu}
(M_u(\phi)f)(\dot{\gamma}):=\phi(r(\gamma))f(\dot{\gamma}), 
\end{equation}
where $\phi\in C_{0}(\gus), \:f\in L^2(G_u /\gudu,\fH{},\sigma)$ 
and $\dot{\gamma}\in G_u /\gudu$. 
Then $M_{u}$ is a representation of $C_{0}(\gus)$ as multiplication operators
on $L^2(G_u /\gudu,\fH{},\sigma)$. \label{MuDfn}

\begin{remark}
We motivate the terminology `multiplication operator'. 
Assume that $\gus/G$ is $T_0$. 
Then by \cite[Theorem 2.1(i))]{Ram90} the map $\Phi:G_u /\gudu\to [u]$,
given by $\Phi(\dot{\gamma}):=r(\gamma)$, is 
a homeomorphism. We can push 
the measure $\sigma_{u}$ forward to get a measure $\sigma_{u}^*$ on $[u]$. Then 
$ L^2(G_u /\gudu,\fH{},\sigma_{u})$ is isomorphic to $L^2([u],\fH{},\sigma_{u}^*)$, where
the isomorphism is given by $f\mapsto f\circ \Phi^{-1}$. Define $N_u :C_{0}(\gus)\to B(L^2([u],\fH{},\sigma_u^{*}))$  by 
$$ (N_u(\phi)g)(v):=\phi(v)g(v).$$ Then $N_u$ is a representation of $C_{0}(\gus)$ as multiplication operators in the conventional sense and
\begin{equation*}
(M_u(\phi)f)(\dot{\gamma})=\phi(r(\gamma))f(\dot{\gamma}) 
						=\phi(r(\gamma))(f\circ\Phi^{-1})(r(\gamma))
						=(N_u(\phi)(f\circ\Phi^{-1}))(r(\gamma)).
\end{equation*} 
Hence $M_u$ is a representation of $C_{0}(\gus)$ as multiplication operators up to isomorphism. 
\hfill $\square$
\end{remark}

Next we describe the *-homomorphism of $C_{0}(\gus)$ into the multiplier algebra 
$\mathcal{M}(C^*(G))$ of $C^*(G)$. By \cite{Ren80} there is a *-homomorphism of $C^*(\gus)$ into $\mathcal{M}(C^*(G))$ (described below). However, note that $C_{0}(\gus)$ is isomorphic to $C^*(\gus)$. To see this, note that if
$f,g\in C_{c}(G)$, then their convolution product $f\ast g$ restricted to 
$\gus$ is just their pointwise product. Then the identity map $\mathrm{id}$ from 
$C_{c}(\gus)\subset C_{0}(\gus)$ to $C_{c}(\gus)\subset C^*(\gus)$ is 
a *-isomorphism which is bounded in the $I$-norm, and hence extends to an isomorphism of $ C_{0}(\gus)$ onto $C^*(\gus)$.
Fix $\phi\in C_{0}(\gus), f\in C_{c}(G)$ and $\gamma\in G$. Define  
\begin{equation} \label{dfnV}
 ((V\phi) f)(\gamma):= \phi(r(\gamma)) f(\gamma), \text{   and}
\end{equation}
$$(f (V\phi))(\gamma):= \phi(s(\gamma)) f(\gamma).$$
Then $V\phi$ acts as double centralizers on $C_{c}(G)$ 
(see the discussion on p.59 of \cite{Ren80}). 
By \cite[Lemma 1.14]{Ren80}  the action of $V\phi$
as double centralizers
extends to $C^*(G)$, and thus the map 
$V:\phi\mapsto V\phi$ is a *-homomorphism of $C_{0}(\gus)$
into  $\mathcal{M}(C^*(G))$.

To show that $L_u\circ V$ is unitarily equivalent to $M_u$ we need an isomorphism \newline $P:\overline{C_{c}(G_u)\odot\fH{}} \to L^2(G_u /\gudu,\fH{},\sigma_{u})$ which intertwines these representations. We define such a $P$ in Lemma \ref{lem:defP} below. For this we use a Borel cross section $c:G_u/\gudu\to G_u$. This means that $c$ is a function that assigns to each equivalence class $\dot{\gamma}\in G_u/\gudu$ a fixed representative in $G_u$ such that $q(c(\dot{\gamma}))=\dot{\gamma}$, where $q:G_u \to \gus/G_u$ is the quotient map. To see that a Borel cross section exists, note that since $G$ is Hausdorff, it follows that $G_u$ closed in $G$. Thus $G_u$ 
is locally compact in addition to being second-countable  and Hausdorff. Therefore $G_u$ is Polish. Then, since the quotient map $q$ is open
(Lemma \ref{OpenQuotientMap}) and 
$q^{-1}(\dot{\gamma})$ is a closed set in $G_u$
for every $\dot{\gamma}\in G_u/\gudu$, it follows that there exists a 
Borel cross section $c:G_u/\gudu\to G_u$, \cite[Theorem 3.4.1]{ArvInvitation}. 
Let $\kappa:G_u \to \gudu$ be the function defined by 
\begin{equation}\label{crossSectionformula}
\dot{\gamma}=c(\dot{\gamma})\kappa(\gamma),
\end{equation} 
for each $\gamma\in G_u$. Note: if $\gamma\in G_u$ and $t\in\gudu$, then 
$\kappa(\gamma t)=\kappa(\gamma)t$.

\begin{lem} \label{lem:defP}
	Let $\omega:\gudu\to B(\fH{})$ be an irreducible unitary representation. Let $f\in C_{c}(G_u), h\in\fH{}$ and $\gamma\in G_u$, let
	$\beta$ be a Haar measure on $\gudu$ and let $\rho_{u}$ be the function given in Lemma \ref{Bruhat}(iii). Define 
	$$ P(f\otimes h)(\dot{\gamma}):= \int_{\gudu}{f(c(\dot{\gamma})t)\:\rho_{u}(c(\dot{\gamma})t)^{-1/2}
		\: \omega_{t}h \:d\beta(t)}. $$
	Then $P$ maps $C_{c}(G_u)\odot\fH{}$ into $L^2(G_u /\gudu,\fH{},\sigma_{u})$. 
\end{lem}
\begin{proof}
We have to show that $P(f\otimes h)$ is measurable and square integrable.
The function $P(f\otimes h)$ is measurable if and only if for all $k\in\fH{}$ the map 
\begin{eqnarray} 
\dot{\gamma}\mapsto \left(P(f\otimes h)(\dot{\gamma}) \mid k \right)_{\fH{}} 
	= \int_{\gudu}{f(c(\dot{\gamma})t)\:\rho_{u}(c(\dot{\gamma})t)^{-1/2} 
	\left(\omega_{t}h \mid k \right)_{\fH{}} \:d\beta(t) } \label{measP}
\end{eqnarray}
is measurable.  To show that (\ref{measP}) is measurable 
we employ Tonelli and Fubini's theorems, \cite[Theorem 8.8]{RudRC}. 
Consider the measure space $(G_u/\gudu\times \gudu,\sigma_u\times\beta)$. 
Define $F:G_u/\gudu\times \gudu \to \mathbb{C}$ by 
\begin{equation} \label{eq:Fmeas}
F(\dot{\gamma},t):= f(c(\dot{\gamma})t)\:\rho_{u}(c(\dot{\gamma})t)^{-1/2} 
	\left(\omega_{t}h \mid k \right)_{\fH{}}.
\end{equation}  
Note that $c$ is a Borel cross section and is thus measurable. Multiplication 
$G_u\times \gudu \to G_u$ is continuous, and thus measurable. 
Hence $(\dot{\gamma},t)\mapsto f(c(\dot{\gamma})t)$ and $(\dot{\gamma},t)\mapsto \rho_{u}(c(\dot{\gamma})t)^{-1/2}$ are both 
measurable functions on $G_u/\gudu\times \gudu$. 
The inner product $\left(\omega_{t}h \mid k \right)_{\fH{}}$ 
is continuous in $t$, and we can view this inner product as the composition
of the projection onto the second coordinate together with a continuous 
function in $t$, which is then also measurable on $G_u/\gudu\times \gudu$. 
The pointwise product of measurable functions is measurable, and so 
$F$ is measurable on $G_u/\gudu\times \gudu$. We show that 
$F\in L^1(\sigma_u\times\beta)$,
so that we can deduce the measurability of (\ref{measP}) from 
Fubini's theorem (\cite[Theorem 8.8(c)]{RudRC}). 
Let $F_{\dot{\gamma}}(t):=F(\dot{\gamma}, t)$.
By Theorem 8.8(b) of \cite{RudRC}, for $F$ to be in $L^1(\sigma_u\times\beta)$, 
it suffices to show that
\begin{eqnarray*}
\int_{G_u/\gudu}{\int_{\gudu}{|F|_{\dot{\gamma}}(t) \:d\beta(t)}
    \:d\sigma_u(\dot{\gamma})} <\infty.
\end{eqnarray*}
Note that 
\begin{eqnarray}
	&& \int_{G_u/\gudu}{\int_{\gudu}{|F|_{\dot{\gamma}}(t) \:d\beta(t)}
    \:d\sigma_u(\dot{\gamma})} \nonumber\\
    &=& \int_{G_u/\gudu}{\int_{\gudu}{|f(c(\dot{\gamma})t)|
    \:\rho_{u}(c(\dot{\gamma})t)^{-1/2} 
	\:|\left(\omega_{t}h \mid k \right)_{\fH{}}| \:d\beta(t)}
    \:d\sigma_u(\dot{\gamma})}   \nonumber\\
    &\leq& ||h||\:||k|| \int_{G_u/\gudu}{\int_{\gudu}{|f(c(\dot{\gamma})t)|
    \:\rho_{u}(c(\dot{\gamma})t)^{-1/2}  \:d\beta(t)}   \:d\sigma_u(\dot{\gamma})}.\label{ineq:integr}
\end{eqnarray}
Since $f\in C_{c}(G_u)$ and $\rho$ are continuous, we have that the pointwise product $|f|\rho_{u}^{1/2}\in C_{c}(G_u)$. 
Apply Lemma \ref{Bruhat}(iv) to (\ref{ineq:integr}) to get 
\begin{eqnarray}
&&\int_{G_u/\gudu}{\int_{\gudu}{|F|_{\dot{\gamma}}(t) \:d\beta(t)}
    \:d\sigma_u(\dot{\gamma})} \nonumber \\
    &\leq& ||h||\:||k|| \int_{G_u/\gudu}{\int_{\gudu}{|f(c(\dot{\gamma})t)|
    \:\rho_{u}(c(\dot{\gamma})t)^{-1/2}  \:d\beta(t)}   \:d\sigma_u(\dot{\gamma})}\nonumber \\
    &=& ||h||\:||k|| \int_{G_u}{|f(\gamma)|\:\rho_{u}(\gamma)^{1/2}\:d\lambda_u(\gamma)}
     < \infty. \label{FubiniApplies}
\end{eqnarray}
Now, by Theorem 8.8(b) of \cite{RudRC}, we have that $F\in L^1(\sigma_u\times\beta)$. 
Then by Theorem 8.8(c) of \cite{RudRC}, the map (\ref{measP}) is measurable.

Next we show square integrability. 
Note that
\begin{eqnarray}
&& \int_{G_u/\gudu}{\left| \left|  P(f\otimes h)(\dot{\gamma}) \right| \right|_{\fH{}} ^{2}  
\: d\sigma_{u}(\dot{\gamma})} \nonumber \\
&=& \int_{G_u/\gudu}{\left| \left|  \int_{\gudu}{f(c(\dot{\gamma})t)
\:\rho_{u}(c(\dot{\gamma})t)^{-1/2}
\: \omega_{t}h \:d\beta(t)} \right| \right|_{\fH{}} ^{2}  \: d\sigma_{u}(\dot{\gamma})} 
\nonumber \\
&<& ||h||_{\fH{}}^{2} \int_{G_u/\gudu}{\left( \int_{\gudu}{\left|  f(c(\dot{\gamma})t)\right|
\:\rho_{u}(c(\dot{\gamma})t)^{-1/2}
\:d\beta(t)}  \right) ^{2}  \: d\sigma_{u}(\dot{\gamma})}.  \label{ineq:PsqIntegr}
\end{eqnarray}
By Lemma \ref{Bruhat}(ii) it follows that 
$$Q(|f|\rho_{u}^{-1/2})(\dot{\gamma}):=  
\int_{\gudu}{\left|  f(c(\dot{\gamma})t)\right|
\:\rho_{u}(c(\dot{\gamma})t)^{-1/2}
\:d\beta(t)}$$
is a function in $C_{c}(G_u/\gudu)$. Then $Q(\:|f|\rho_{u}^{-1/2}\:)^2$
is also in $C_{c}(G_u/\gudu)$.
Now (\ref{ineq:PsqIntegr}) is equal to 
$$ ||h||_{\fH{}}^{2} \int_{G_u/\gudu}{ Q(\:|f|\rho_{u}^{-1/2}\:)^2
 \: d\sigma_{u}(\dot{\gamma})} <\infty,$$
which shows that $P(f\otimes h)$ is square integrable. This proves our claim. 
\end{proof}

Recall that $V:C_{0}(\gus)\to\mathcal{M}(C^*(G))$ is 
a *-homomorphism and that \newline
$L_{u}: \mathcal{M}(C^*(G)) \to B(\overline{C_{c}(G_u)\odot\fH{}})$ denotes
the extension of $\ind{\gudu}{G}{\pi_{\omega}}$ to $\mathcal{M}(C^*(G))$.

\begin{prop} \label{mainEquivReps}
Suppose that $G$ is a second-countable locally compact Hausdorff groupoid with 
a Haar system $\{\lambda^u\}_{u\in\gus}$. Fix $u\in\gus$ and suppose that 
$\omega:\gudu\to B(\fH{})$ is an irreducible unitary representation and 
$\pi_{\omega}$ its integrated form. 
Then $M_u$ is unitarily equivalent to $L_u\circ V$.
\end{prop}
\begin{proof}
Let $\rho_u, \Delta_u$ and $\sigma_{u}$ be as is in Lemma \ref{Bruhat}.
To simplify notation we suppress the subscript $u$. 
Let $f\in C_{c}(G_u), h\in\fH{}$ and $\gamma\in G_u$. Then by Lemma \ref{lem:defP}  
$$ P(f\otimes h)(\dot{\gamma}):= \int_{\gudu}{f(c(\dot{\gamma})t)\:\rho(c(\dot{\gamma})t)^{-1/2}
    \: \omega_{t}h \:d\beta(t)}, $$
defines a  map  
$P:C_{c}(G_u)\odot\fH{} \to L^2(G_u /\gudu,\fH{},\sigma)$, where $\beta$ is the Haar on $\gudu$. We claim 
that $P$ isometric and maps onto a dense subspace, so that it extends to a unitary operator of the completion 
$\overline{C_{c}(G_u)\odot\fH{}}$ onto $L^2(G_u /\gudu,\fH{},\sigma)$.

We first show that $P$ is isometric on $C_{c}(G_u)\odot\fH{}$.
Let $f,g\in C_{c}(G_u) $ and $h,k\in\fH{}$. We need to show that 
$$(P(f\otimes h)\mid P(g\otimes k))_{L^2}= \left( f\otimes h \mid g\otimes k \right).$$
We break up the calculations so that they are easier to follow. 
Note that for $\dot{\gamma}\in G_u /\gudu$
\begin{eqnarray}
&&(P(f\otimes h)(\dot{\gamma})\mid 
			P(g\otimes k)(\dot{\gamma}))  \nonumber \\
\hspace{1.5cm} &=&  \left( \int_{\gudu}{f(c(\dot{\gamma})t)\rho(c(\dot{\gamma})t)^{-1/2}
\omega_th \:d\beta(t)}   \mid  
\int_{\gudu}{g(c(\dot{\gamma})s)\rho(c(\dot{\gamma})s)^{-1/2}\omega_sk \:d\beta(s)} \right)
\nonumber  \\
&=&\int_{\gudu}{\int_{\gudu}{
  f(c(\dot{\gamma})t)\overline{g(c(\dot{\gamma})s)}\:\rho(c(\dot{\gamma})t)^{-1/2} 
  \rho(c(\dot{\gamma})s)^{-1/2}
\left(\omega_th \mid \omega_sk   \right)
  \:d\beta(t)}\:d\beta(s)} \nonumber  \\
&& \text{(now apply Lemma \ref{Bruhat}(iii))} \nonumber  \\
&=& \int_{\gudu}{\int_{\gudu}{
  f(c(\dot{\gamma})t)\overline{g(c(\dot{\gamma})s)}\:\rho(c(\dot{\gamma}))^{-1}  
  \Delta(t)^{-1/2} \Delta(s)^{-1/2} \left(\omega_{s^{-1}t}h \mid k   \right)
  \:d\beta(t)}\:d\beta(s)}. \label{LongCal1}
\end{eqnarray}
Substitute $r=s^{-1}t$ and apply Lemma \ref{Bruhat}(iii) so that 
$$\rho(c(\dot{\gamma}))^{-1}\Delta(t)^{-1/2} \Delta(s)^{-1/2} 
=\rho(c(\dot{\gamma}))^{-1}\Delta(r)^{-1/2}\Delta(s)^{-1}
=\rho(c(\dot{\gamma})s)^{-1}\Delta(r)^{-1/2}.$$ 
Then (\ref{LongCal1}) is
\begin{eqnarray*}
=\int_{\gudu}{\int_{\gudu}{
  f(c(\dot{\gamma})sr)\overline{g(c(\dot{\gamma})s)}\:\rho(c(\dot{\gamma})s)^{-1}  
  \Delta(r)^{-1/2} \left(\omega_rh \mid k   \right)
  \:d\beta(r)}\:d\beta(s)}. 
\end{eqnarray*}

Then 
\begin{eqnarray}
&&(P(f\otimes h)\mid P(g\otimes k))_{L^2} \nonumber \\
&=& \int_{G_u /\gudu}{(P(f\otimes h)(\dot{\gamma})\mid 
			P(g\otimes k)(\dot{\gamma})) \:d\sigma(\dot{\gamma})}\nonumber   \\
&=& \int_{G_u /\gudu}{\int_{\gudu}{\int_{\gudu}{
  f(c(\dot{\gamma})sr)\overline{g(c(\dot{\gamma})s)}\:\rho(c(\dot{\gamma})s)^{-1}  
  \Delta(r)^{-1/2}}}} \nonumber \\
	&& \hspace{6cm}  \left(\omega_rh \mid k   \right)
 	    \:d\beta(r)\:d\beta(s)   \:d\sigma(\dot{\gamma})  \nonumber \\
&=& \int_{\gudu}{\int_{G_u /\gudu}{\int_{\gudu}{
  f(c(\dot{\gamma})sr)\overline{g(c(\dot{\gamma})s)}\:\rho(c(\dot{\gamma})s)^{-1}  
  \Delta(r)^{-1/2}} }} \label{LongCal2} \\
	&&  \hspace{6cm}  \left(\omega_rh \mid k   \right)
  \:d\beta(s)   \:d\sigma(\dot{\gamma})\:d\beta(r). \nonumber \:\:\:\:\:\:
  \:\:\:\:\:\:   
  \end{eqnarray}
Focussing on the inner two integrals we can view $f$ as a function $f_r$
  of $c(\dot{\gamma})$ and $s$ only, with $r$ fixed; i.e.
  $f_r(c(\dot{\gamma})s):=f(c(\dot{\gamma})sr)$. Then 
  (\ref{LongCal2}) is equal to
\begin{eqnarray*}
 &=& \left(\int_{\gudu}{\int_{G_u /\gudu}{\int_{\gudu}{
  f_r(c(\dot{\gamma})s)\overline{g(c(\dot{\gamma})s)}\rho(c(\dot{\gamma})s)^{-1}  
  \Delta(r)^{-1/2} \omega_rh  \:d\beta(s)}   \:d\sigma(\dot{\gamma})\:d\beta(r)}}
  \mid k \right) \\
&& \text{(now apply Equation \ref{crossSectionformula} to substitute $c(\dot{\gamma})$
  by $\gamma\kappa(\gamma)^{-1}$ )} \\
&=& \left(\int_{\gudu} {\left( \int_{G_u /\gudu}{\int_{\gudu}{
  f_r(\gamma\kappa(\gamma)^{-1} s)
  \overline{g(\gamma\kappa(\gamma)^{-1} s)}
  \rho(\gamma\kappa(\gamma)^{-1} s)^{-1}  
    \:d\beta(s)}   \:d\sigma(\dot{\gamma})}\right) \Delta(r)^{-1/2} \omega_rh\:d\beta(r)}
  \mid k \right) \\
&=& \left(\int_{\gudu} {\left( \int_{G_u /\gudu}{\int_{\gudu}{
  f_r(\gamma t)
  \overline{g(\gamma t)}
  \rho(\gamma t)^{-1}  
    \:d\beta(t)}   \:d\sigma(\dot{\gamma})}\right) \Delta(r)^{-1/2} \omega_rh\:d\beta(r)}
  \mid k \right) \\
&& \text{(now apply Lemma \ref{Bruhat}(iv) to the two inner integrals)} \\
&=& \left(\int_{\gudu} {\left( \int_{G_u}{
  f_r(\gamma)
  \overline{g(\gamma)} 
    \:d\lambda_u(\gamma) }\right) \Delta(r)^{-1/2} \omega_rh\:d\beta(r)}
  \mid k \right) \\
&& \text{(now apply Equation \ref{eq:ModuleInnerProd}  
to the inner integral)} \\
&=& \left(\int_{\gudu} {
  \langle g,f\rangle_{_{C^*(\gudu)}} (r)\:\Delta(r)^{-1/2} \:\omega_rh\:d\beta(r)}  
  \mid k \right) \\
&=& \left( \pi_{\omega}(\langle g,f\rangle_{_{C^*(\gudu)}})\: h \mid k \right) 
\hspace{0.6cm}\text{ (recall $\pi_{\omega}$ is given by Equation 
\ref{eq:IntergfromGeneral})} \\
&=& \left( f\otimes h \mid g\otimes k \right) \hspace{1cm}\text{ (by Equation
\ref{eq:InduedHilbertInner})}.
\end{eqnarray*}
Hence $(P(f\otimes h)\mid P(g\otimes k))_{L^2}=\left( f\otimes h \mid g\otimes k \right)$, showing $P$ is isometric.

Next we show that $P$ maps onto a dense subspace of $L^2(G_u /\gudu,\fH{},\sigma)$. For 
this part of the proof we adapt techniques from
Proposition 5.4 in \cite{CrossedProd} for transformation groups 
to our groupoid setting.
Suppose that $F\in C_{c}(G_u /\gudu,\fH{})$ and $\epsilon>0$. Then it 
will suffice to show that there is an $f$ in the image $\mathrm{Im}(P)$ of $P$
such that 
\begin{equation} \label{UnitaryOntoIneq}
||f-F||_{L^2} < \epsilon^2.
\end{equation}
Let $D:=\supp(F)$ and let $C$ be a compact neighbourhood of $D$. We need the following lemma:

\begin{lem} \label{FgammainImage}
Fix $\dot{\gamma}\in D$. Then there is a $F_{\gamma}\in C_{c}(G_u /\gudu,\fH{})$
such that $F_{\gamma}\in \mathrm{Im}(P)$ and 
\begin{equation} \label{POntoLemmaIneq}
||\: F_{\gamma}(\dot{\gamma})-F(\dot{\gamma})\: ||_{\fH{} }
< \frac{\epsilon}{3\sigma(C)^{1/2}}.
\end{equation} 
\end{lem}
\begin{proof}
Since $t\mapsto \omega_{t}(F(\dot{\gamma}))$ is continuous there is a 
neighbourhood $N_u$ of (the fixed) $u$ in $\gudu$ such that 
	\begin{equation} \label{IneqOmega}
	||\: \omega_{t}(F(\dot{\gamma}))-F(\dot{\gamma})\: ||_{\fH{} }
	< \frac{\epsilon}{3\sigma(C)^{1/2}},
	\end{equation}
whenever $t\in N_u$.

Let $N_{c(\dot{\gamma})}$ be a neighbourhood of $c(\dot{\gamma})$ in $G_u$, and
let $N_{c(\dot{\gamma})}N_u=\{\eta t: \eta\in N_{c(\dot{\gamma})},t\in N_u\}$.
Let $g\in C_{c}(G_u)^{+}$ be such that 
$\supp(g)\subset N_{c(\dot{\gamma})}N_u$, $g(c(\dot{\gamma}))\neq 0$ and 
	\begin{equation} \label{normalisedg}
	\int_{\gudu}{g(c(\dot{\gamma})t)\rho(c(\dot{\gamma})t)^{-1/2}}\:d\beta(t) = 1.
	\end{equation}
Then  for $k\in\fH{}$ with $||k||\leq $1, we have
 \begin{eqnarray*}
 &&  \mid \left(P(g\otimes F(\dot{\gamma}))(\dot{\gamma})\mid k) \right)_{\fH{}}-
 \left(F(\dot{\gamma}) \mid k\right)_{\fH{}}  \mid \\
 &=& \left| \left(\int_{\gudu}{g(c(\dot{\gamma})t)\rho(c(\dot{\gamma})t)^{-1/2} \: 
     \omega_{t}(F(\dot{\gamma}))\:d\beta(t)}
     - F(\dot{\gamma}) \mid k\right)_{\fH{}}  \right| \\
 &=& \left|\left( \int_{\gudu}{g(c(\dot{\gamma})t)\rho(c(\dot{\gamma})t)^{-1/2} \: 
     \omega_{t}(F(\dot{\gamma}))\:d\beta(t)}
     -   \left(\int_{\gudu}{g(c(\dot{\gamma})t)\rho(c(\dot{\gamma})t)^{-1/2}}\:d\beta(t)\right)
     F(\dot{\gamma}) \mid k\right)_{\fH{}}  \right|  \\
 &=& \left| \int_{\gudu}{g(c(\dot{\gamma})t)\rho(c(\dot{\gamma})t)^{-1/2} \: 
     \left( \omega_{t}(F(\dot{\gamma})) -F(\dot{\gamma})\mid k\right)_{\fH{}} 
     \:d\beta(t)} \right| \\
 &\leq &  \int_{\gudu}{g(c(\dot{\gamma})t)\rho(c(\dot{\gamma})t)^{-1/2} \: 
     || \omega_{t}(F(\dot{\gamma})) -F(\dot{\gamma})||_{\fH{}} \: ||k||_{\fH{}}
      \:d\beta(t)} \\
 &< &  \frac{\epsilon}{3\sigma(C)^{1/2}}, \hspace{1cm} 
 \text{ (by applying (\ref{IneqOmega}) and (\ref{normalisedg})).}
 \end{eqnarray*} 
Thus putting 
$$F_{\gamma}:=P(g\otimes F(\dot{\gamma}))$$
gives the desired inequality (\ref{POntoLemmaIneq}), and then $F_{\gamma}\in \mathrm{Im}(P)$.

We show that $F_{\gamma}$ is continuous with compact support. We first 
show continuity. 
Note: for $g$ as in (\ref{normalisedg}) we have 
$$\int_{\gudu}{g(c(\dot{\eta})t)\rho(c(\dot{\eta})t)^{-1/2}}\:d\beta(t)
=\int_{\gudu}{g(\zeta t)\rho(\zeta t)^{-1/2}}\:d\beta(t),$$
for all $\zeta\in \dot{\eta}$, because then 
$c(\dot{\eta})\gudu=\eta\gudu=\zeta\gudu$. 
Suppose that $\dot{\eta_i}\to \dot{\eta}$ in $G_u /\gudu$ and $\delta>0$. 
We show that $F_{\gamma}(\dot{\eta_i})\to F_{\gamma}(\dot{\eta})$.
Since (i) 
$G_u$ is a right $\gudu$-space, (ii) the quotient map $q_{u}:G_u\to G_u/\gudu$ is open by 
Lemma \ref{OpenQuotientMap}, and (iii)
$\{\eta_i\}$ is a sequence such that $q_{u}(\eta_i)=\dot{\eta_i}\to q_u(\eta)=\dot{\eta}$,
it follows from Lemma 3.38 in \cite{CrossedProd} that there is a subsequence $\{\eta_{i_j}\}$ and a sequence $\{s_{i_j}\}\subset \gudu$ such that $\eta_{i_j}s_{i_j}\to \eta$. We will only 
use the subsequence. Thus we relabel all $i_j$-indices to $i$ only. Then since $g$
and $\rho$ are continuous, so is their pointwise product. Hence
 there is an $i_0$ such that for all $i>i_0$ we have
 $$\left| g(\eta_{i}s_{i}t)\:\rho(\eta_{i}s_{i}t)^{-1/2}
	- g(\eta t)\:\rho(\eta t)^{-1/2}\right|  < \frac{\delta}{
	||F(\dot{\gamma})||_{\fH{}}\:\beta(\supp(g))}.$$
Then 
$$  \int_{\gudu}{\left|
	g(\eta_{i}s_{i}t)\:\rho(\eta_{i}s_{i}t)^{-1/2}
	- g(\eta t)\:\rho(\eta t)^{-1/2}\right| \:d\beta(t)}  
	< \frac{\delta}{||F(\dot{\gamma})||_{\fH{}}}. $$
Then 
\begin{eqnarray*}
&& || F_{\gamma}(\dot{\eta_i})-F_{\gamma}(\dot{\eta}) ||_{\fH{}} \\
&=&\left|\left|\: \int_{\gudu}{g(c(\dot{\eta_i})t)\:\rho(c(\dot{\eta_i})t)^{-1/2}
	\omega_{t}(F(\dot{\gamma}))\: d\beta(t)}
	- \int_{\gudu}{g(c(\dot{\eta})t)\rho(c(\dot{\eta})t)^{-1/2}
	\omega_{t}(F(\dot{\gamma}))\: d\beta(t)} \:\right|\right|_{\fH{}} \\
&=&\left|\left|\: \int_{\gudu}{g(c(\dot{\eta_{i}s_{i}})t)\rho(c(\dot{\eta_{i}s_{i}})t)^{-1/2}
	\omega_{t}(F(\dot{\gamma}))\: d\beta(t)}
	- \int_{\gudu}{g(c(\dot{\eta})t)\rho(c(\dot{\eta})t)^{-1/2}
	\omega_{t}(F(\dot{\gamma}))\: d\beta(t)} \:\right|\right|_{\fH{}} \\
&=&\left|\left|\: \int_{\gudu}{g(\eta_{i}s_{i}t)\:\rho(\eta_{i}s_{i}t)^{-1/2}
	\omega_{t}(F(\dot{\gamma}))\: d\beta(t)}
	- \int_{\gudu}{g(\eta t)\:\rho(\eta t)^{-1/2}
	\omega_{t}(F(\dot{\gamma}))\: d\beta(t)} \:\right|\right|_{\fH{}} \\
&=&\left|\left|\: \int_{\gudu}{\left(g(\eta_{i}s_{i}t)\:\rho(\eta_{i}s_{i}t)^{-1/2}
	- g(\eta t)\:\rho(\eta t)^{-1/2} \right)
	\omega_{t}(F(\dot{\gamma}))\: d\beta(t)} \:\right|\right|_{\fH{}} \\
&\leq & ||F(\dot{\gamma})||_{\fH{}}\: \int_{\gudu}{\left|
	g(\eta_{i}s_{i}t)\:\rho(\eta_{i}s_{i}t)^{-1/2}
	- g(\eta t)\:\rho(\eta t)^{-1/2}\right| 
	\:d\beta(t)} \\
&<& \delta.
\end{eqnarray*}
Thus $F_{\gamma}$ is continuous. 

Now we consider the $\supp(F_{\gamma})$. Since the quotient map 
$q_u:G_u\to G_u/\gudu$ is continuous, 
it follows that $q_u(\supp(g))$ is compact in $G_u/\gudu$. Since 
$q_u(\supp(g))$ is compact and 
$G_u/\gudu$ is Hausdorff (Lemma \ref{quotientHausdorff}), it 
follows that $q_u(\supp(g))$ is closed in $G_u/\gudu$. Thus, to 
show that $\supp(F_{\gamma})$ is compact, it will
suffice to show that  $\supp(F_{\gamma})\subset q_u(\supp(g))$.
Suppose that $\dot{\eta}$ is such that $F_{\gamma}(\dot{\eta})\neq 0$. Then 
$$F_{\gamma}(\dot{\eta})=\int_{\gudu}{g(c(\dot{\eta})t)\rho(c(\dot{\eta})t)^{-1/2}
	 \: \omega_{t}(F(\dot{\gamma}))\:d\beta(t)} \neq 0.$$
Since $g$ and $\rho$ are positive and non-zero, it follows that 
$g(c(\dot{\eta})t)\rho(c(\dot{\eta})t)^{-1/2}\neq 0$
for some $t\in \gudu$. That is, $c(\dot{\eta})t\in\supp(g)$ for some $t\in \gudu$. Note 
$q_u(c(\dot{\eta})t)=q_u(c(\dot{\eta}))=\dot{\eta}$. Thus 
$\supp(F_{\gamma})\subset q_u(\supp(g))$, and so $F_{\gamma}$ has 
compact support. Hence $F_{\gamma}$ has all the required properties and the lemma's proof is finished.
\end{proof}

We now continue with the proof of Proposition \ref{mainEquivReps}.
By Lemma \ref{FgammainImage} we can find for every $\dot{\gamma}\in\supp(F)$ a 
$$F_{\gamma}=P(g_{\dot{\gamma}}\otimes F(\dot{\gamma}))\in C_{c}(G_u /\gudu,\fH{})$$
such that 
	\begin{equation} \label{Ineq1of3}
	||\: F_{\gamma}(\dot{\gamma})-F(\dot{\gamma})\: ||_{\fH{} }
	< \frac{\epsilon}{3\sigma(C)^{1/2}}.
	\end{equation}
Since $F$ and every $F_{\gamma}$ are continuous, there 
is for every $\dot{\gamma}\in D$ an open neighbourhood $N_{\dot{\gamma}}$ 
of $\dot{\gamma}$ such that $N_{\dot{\gamma}}\subset C$, 
and for all $\dot{\eta}\in N_{\dot{\gamma}}$ we have
	\begin{equation}  \label{Ineq2of3}
	||\: F_{\gamma}(\dot{\eta})-F_{\gamma}(\dot{\gamma})\: ||_{\fH{} }
	\leq  \frac{\epsilon}{3\sigma(C)^{1/2}},
	\end{equation}
and 
	\begin{equation}  \label{Ineq3of3}
	||\: F(\dot{\gamma})-F(\dot{\eta})\: ||_{\fH{} }
	\leq  \frac{\epsilon}{3\sigma(C)^{1/2}}.
	\end{equation}

Applying inequalities (\ref{Ineq1of3}), (\ref{Ineq2of3}) and (\ref{Ineq3of3}) we 
see that, for all $\dot{\eta}\in N_{\dot{\gamma}}$,
\begin{eqnarray*}
 ||\: F_{\gamma}(\dot{\eta})-F(\dot{\eta})\: ||_{\fH{} } 
&\leq & ||\: F_{\gamma}(\dot{\eta})-F_{\gamma}(\dot{\gamma})\: ||_{\fH{} } +
		||\: F_{\gamma}(\dot{\gamma})-F(\dot{\gamma})\: ||_{\fH{} } +
		||\: F(\dot{\gamma})-F(\dot{\eta})\: ||_{\fH{} }  \\
&<&  \frac{\epsilon}{\sigma(C)^{1/2}}.
\end{eqnarray*}
Since $D$ is compact and contained in $\cup_{\dot{\gamma}\in D}N_{\dot{\gamma}}$,
there are $\dot{\gamma}_{1}, \dot{\gamma}_{2}, \ldots, \dot{\gamma}_{n}$ such that 
$D\subset\cup_{i=1}^{n}N_{\dot{\gamma_{i}}}$. Also, since $G_u/\gudu$ 
is Hausdorff and every $N_{\dot{\gamma_{i}}}\subset C$, it follows that every 
$N_{\dot{\gamma_{i}}}$ has compact closure. 
Then there is a partition of unity  $\{\psi_{i}\}_{i=1}^{n}$ subordinate to 
$\{N_{\gamma_{i}}\}_{i=1}^{n}$ such that for $i=1,2,\ldots, n$
\begin{enumerate}
\item[(i)] $\psi_{i}\in C_{c}(G_u/\gudu),$
\item[(ii)] $0\leq \psi_{i}(\dot{\gamma}) \leq 1$ for all $\dot{\gamma}\in G_u/\gudu$,
\item[(iii)] $\supp(\psi_{i})\subset N_{\dot{\gamma}}$,
\item[(iv)] $\Sigma_{i=1}^{n} \psi(\dot{\gamma})=1$ for all $\dot{\gamma}\in D$, and
\item[(v)] $\Sigma_{i=1}^{n} \psi(\dot{\gamma})\leq1$ for all $\dot{\gamma}\notin D$
\end{enumerate} 
Let $\dot{\gamma}\in D$ and suppose that $\dot{\gamma}\in N_{\gamma_{i}}$.
Let $g_{\dot{\gamma}}$
be as in (\ref{normalisedg}). For $\eta\in G_u$ put
$$g_{\dot{\gamma}}'(\eta):=\psi_{i}(q_u(\eta))g_{\dot{\gamma}}(\eta).$$
Then $g_{\dot{\gamma}}'\in C_{c}(G_u /\gudu)$ 
and 
$$(\psi_{i} F_{\gamma_{i}})(\dot{\eta})=P(g_{\dot{\gamma}}'\otimes F(\dot{\gamma_{i}}))(\dot{\eta}).$$
Hence $\psi_{i} F_{\gamma_{i}}\in \textrm{Im}(P)$, and so is 
$\Sigma_{i=1}^{n}\psi_{i} F_{\gamma_{i}}$.
Then for all $\dot{\eta}\in D$
\begin{eqnarray}
\left|\left| \:\Sigma_{i=1}^{n}\psi_{i}(\dot{\eta})F_{\gamma_{i}}(\dot{\eta})-
    F(\dot{\eta}) \right|\right|_{\fH{}} 
&=& \left|\left| \:\Sigma_{i=1}^{n}\psi_{i}(\dot{\eta})F_{\gamma_{i}}(\dot{\eta})-
    \Sigma_{i=1}^{n}\psi_{i}(\dot{\eta})F(\dot{\eta}) \:\right|\right|_{\fH{}}  \nonumber\\
&=& \left|\left|\:\left(\: \Sigma_{i=1}^{n}\psi_{i}(\dot{\eta})\:\right)\:
	\left(\: F_{\gamma_{i}}(\dot{\eta})- F(\dot{\eta})\:\right) 
	\:\right|\right|_{\fH{}}\nonumber \\
&= & \left|\left|F_{\gamma_{i}}(\dot{\eta})- F(\dot{\eta})\:\right|\right|_{\fH{}}\nonumber \\
&<& \frac{\epsilon}{\sigma(C)^{1/2}} \nonumber.
\end{eqnarray}
Since $\Sigma_{i=1}^{n}\psi_{i}(\dot{\eta})F_{\gamma_{i}}(\dot{\eta})-
F(\dot{\eta})$  vanishes outside of $C$ for every $\dot{\eta}\in D$, we have that
\begin{eqnarray*}
\left|\left| \:\Sigma_{i=1}^{n}\psi_{i}F_{\gamma_{i}}-F \right|\right|_{L^{2}}^2 
&=& \int_{G_u/\gudu}{\left|\left| \:\Sigma_{i=1}^{n}
	\psi_{i}(\dot{\eta})F_{\gamma_{i}}(\dot{\eta})-
    F(\dot{\eta}) \right|\right|_{\fH{}}^2 \: d\sigma(\dot{\eta}) } \\
    &<& \left(\frac{\epsilon}{\sigma(C)^{1/2}}\right)^2 \sigma(C) \\
    &=& \epsilon^2.
\end{eqnarray*}
Thus $f:=\Sigma_{i=1}^{n}\psi_{i}F_{\gamma_{i}}$ satisfies 
(\ref{UnitaryOntoIneq}). We have shown that $P$ is isometric and maps
onto a dense subspace. Hence we can  extend $P$ to a unitary operator 
from $\overline{C_{c}(G_u)\odot\fH{}}$ onto $L^2(G_u /\gudu,\fH{},\sigma)\overline{}$.

Lastly we show that $P$ intertwines the representations
$M_u$ and $L_u\circ V$.
Let $f\in C_{c}(G), g\in C_{c}(G_u)$, $\phi\in C_{0}(\gus)$ and $ h\in\fH{}$.
Recall that $\ind{\gudu}{G}{\pi_{\omega}}$ is defined through an 
action of $C_{c}(G)$ on the right $C_c(\gudu)$-pre-Hilbert module
$C_{c}(G_u)$.
 Specifically  the action is given by $f\cdot g= f\ast g$, and extends to give a 
 representation of 
$C^*(G)$. The representation $\ind{\gudu}{G}{\pi_{\omega}}$ acts on the completion of $C_{c}(G_u)\odot\fH{}$  with respect to 
the inner product
$$ (f\otimes h \mid g\otimes k) = (\pi_{\omega}(\langle g,f \rangle_{u}) h\mid k) .$$
The action is given by 
$$\ind{\gudu}{G}{\pi_{\omega}}(f)(g\otimes h)= f\cdot g \otimes h =f\ast g\otimes h.  $$ 
The extension $L_u$ of $\ind{\gudu}{G}{\pi_{\omega}}$ to $\mathcal{M}(C^*(G))$
is characterized by 
$$L_u(m)(f\cdot g \otimes h)=L_u(mf)(g\otimes h)=
\ind{\gudu}{G}{\pi_{\omega}}(mf)(g\otimes h)=(mf)\cdot g\otimes h,$$
for all $m\in \mathcal{M}(C^*(G))$. 
Recall that 
$$((V\phi)f)(\gamma)=\phi(r(\gamma))f(\gamma).$$
We want to show that 
\begin{equation} \label{intertwine}
 PL_u(V\phi)=M_u(\phi)P.
\end{equation}
Because the action of $C^*(G)$ on $\overline{C_{c}(G_u)}$ is non-degenerate, the 
set 
$$\{a\cdot x:a\in C^*(G), x\in \overline{C_{c}(G_u)} \} $$
is dense in $\overline{C_{c}(G_u)}$. By the continuity of the action,
$$\{f\cdot g:f\in C_{c}(G), g\in C_{c}(G_u) \} $$
is dense in $\overline{C_{c}(G_u)}$. It follows that for 
$f\in C_{c}(G), g\in C_{c}(G_u)$ and $ h\in\fH{}$, the elementary 
tensors $f\cdot g\otimes h$ span a dense subset of  
$\overline{C_{c}(G_u)\odot\fH{}}$. Hence, by the continuity of these operators, to see that 
$$ PL_u(V\phi)=M_u(\phi)P, $$
 it suffices to check that
$$(PL_u (V\phi))(f\cdot g\otimes h) 
=M_u (\phi)P(f\cdot g\otimes h).$$ 
Let $\dot{\gamma}\in G_u/\gudu$. Then 
\begin{eqnarray}
(PL_u (V\phi))(f\cdot g\otimes h)(\dot{\gamma}) 
 &=& P(L_u(V\phi))(f\cdot g\otimes h)(\dot{\gamma}) \nonumber \\
 &=& P(L_u(V\phi f)) (g\otimes h)(\dot{\gamma})\nonumber  \\
 &=& P((V\phi f)\cdot g\otimes h)(\dot{\gamma})\nonumber  \\
 &=& \int_{\gudu}{ ((V\phi f)\cdot g)(c(\dot{\gamma})t)\:\rho((c(\dot{\gamma})t))^{-1/2}
 \:\omega_t h \:d\beta(t)}  \nonumber\\
  &=& \int_{\gudu}{ ((V\phi f)\ast g)(c(\dot{\gamma})t)\:\rho((c(\dot{\gamma})t))^{-1/2}
 \:\omega_t h \:d\beta(t)}. \label{eq:Intertwine1}
\end{eqnarray}
Note that if $\eta\in G^{r(c(\dot{\gamma})t)}$ then 
$$r(\eta)=s(\eta^{-1})= r(c(\dot{\gamma})t)=r(\gamma).$$
 Then we have that 
\begin{eqnarray}
((V\phi f)\ast g)(c(\dot{\gamma})t) &=&  
\int_{G}{(V\phi f)(\eta) g(\eta^{-1}c(\dot{\gamma})t) \:d\lambda^{r(c(\dot{\gamma})t)}}
\nonumber \\
&=& \int_{G}{\phi(r(\eta)) f(\eta) g(\eta^{-1}c(\dot{\gamma})t) 
\:d\lambda^{r(c(\dot{\gamma})t)}} \nonumber\\
&=& \phi(r(\gamma))\int_{G}{ f(\eta) g(\eta^{-1}c(\dot{\gamma})t) 
\:d\lambda^{r(c(\dot{\gamma})t)}} \nonumber\\
&=& \phi(r(\gamma))\: (f\ast g)(c(\dot{\gamma})t) \label{eq:Intertwine2}.
\end{eqnarray} 
Now by  (\ref{eq:Intertwine2}), we have (\ref{eq:Intertwine1}) 
\begin{eqnarray*}
 &=& \int_{\gudu}{ \phi(r(\gamma))\: f\ast g(c(\dot{\gamma})t)
 (c(\dot{\gamma})t)\:\rho((c(\dot{\gamma})t))^{-1/2}
 \:\omega_t h \:d\beta(t)} \nonumber \\
 &=& \phi(r(\gamma))\int_{\gudu}{ (f\cdot g)(c(\dot{\gamma})t)
 \:\rho((c(\dot{\gamma})t))^{-1/2} \:\omega_t h \:d\beta(t)}  \nonumber \\
 &=&\phi(r(\gamma)) P((f\cdot g)\otimes h)(\dot{\gamma}) \\
 &=& M_u (\phi)P((f\cdot g)\otimes h)(\dot{\gamma}) \label{intertwinedensesubspace}.
\end{eqnarray*}
Hence $P$ intertwines the representations on a dense subspace. By continuity
we have $PL_u (V\phi)=M_u (\phi)P$, and the proof is finished.
\end{proof}

\section{CHARACTERIZING CCR GROUPOID $C^*$-ALGEBRAS} \label{sec:CCRmain}
In this section we prove our main theorem, Theorem \ref{mainCCRthm},  which 
shows that if a groupoid $C^*$-algebra is CCR, then the orbit space of the groupoid
is $T_1$ and the stabilizers of $G$ are CCR. We use the following following lemma.

\begin{lem} \label{ccrrepcontainiseqv}
	Let $\ca{A}$ be a CCR $C^*$-algebra. Suppose that $\pi,\tau$ are non-zero 
	irreducible representations
	of $\ca{A}$ acting on $\fH{\pi}$ and $\fH{\tau}$, respectively. 
	If $\ker\pi\subseteq\ker\tau$, then
	$\pi$ is unitarily equivalent to $\tau$.
\end{lem}
  
  \begin{proof}
	Since  $\ca{A}$ is CCR and $\pi$ and $\tau$ are non-zero 
	irreducible representations, we have that $\ker\pi$ and 
	$\ker\tau$ are maximal closed two-sided ideals in $\ca{A}$ by 
	\cite[Corollary 4.1.11]{DixC}. 
	Thus, $\ker\pi\subseteq\ker\tau$ implies that $\ker\pi=\ker\tau$. 
	
	Because  $\ca{A}$ is CCR it is also GCR. Since $\ker\pi=\ker\tau$ 
	and  $\pi$ and $\tau$  are irreducible, it follows that
	 $\pi$ is unitarily equivalent to $\tau$
	by \cite[Theorem 4.3.7]{DixC}.
  \end{proof}

We fix some notation. Let $u\in\gus$. The trivial representation 
$1_u:\gudu\to \mathbb{C}$ is given by $1_u(t)=1$, for every $t\in \gudu$. If
$\pi_{1_u}:C^*(\gudu) \to \mathbb{C}$ 
denotes the integrated form of the $1_u$, then we let 
$$l^u:=\ind{\gudu}{G}{\pi_{1_u}}$$
denote the induced representation of $C^*(G)$. 
Similarly, if 
$C$ is any subgroup in $G$, say $C\subset \gudu$, we let $l^C$ denote the 
 representation of $C^*(G)$ 
induced from the (integrated form of the) trivial representation of $C$. Note: 
since $l^C$ is not induced from a stabilizer, it may not be 
irreducible.

Lemma \ref{convInducedTrivialReps} below essentially appears as part the proofs of  
Clark's Lemma 5.4 in  \cite{Cla07} for groupoids and Williams'
Lemma 4.9 in \cite{Wil81} for transformation groups.
Clark and Williams both prove a continuity result which uses an assumption that the stabilizers are amenable.  Here, in Lemma \ref{convInducedTrivialReps}, we separate out the part of their proofs that do not use amenability of the stabilizers.

\begin{lem} \label{convInducedTrivialReps}
Suppose that $\{u_{n}\}\subset[u]$ and $u_n\to v$ in $\gus$. Then there exists a closed subgroup $C\subset\gud{v}{v}$ such that $\ker(l^{u_n})\subset\ker(l^C)$ for every $n$.
\end{lem}
\begin{proof}
Since $\{u_{n}\}\subset[u]$, we have that $l^{u_n}\sim l^{u_m}$ for all $m,n\in\mathbb{N}$, \cite[Lemma 5.1]{Cla07}. Let $\Sigma$ denote the set of all closed subgroups of $G$. Then 
$\Sigma \cup \{\emptyset\}$ is a compact Hausdorff space in the
Fell topology. 
Thus $\{\gud{u_n}{u_n}\}$ has a convergent subsequence. Say, after 
relabelling, that $\gud{u_n}{u_n}\to C$ in $\Sigma \cup \{\emptyset\}$.
We claim that $C\neq\emptyset$ and that  $C\subset\gud{v}{v}$.
To see this, recall from the characterization 
of convergence in the Fell topology (Lemma H.2 in \cite{CrossedProd}) that since 
$\gud{u_n}{u_n}\to C$ and $u_n\in  \gud{u_n}{u_n}$ with $u_n\to v$, 
it follows that $v\in C$.  
Thus $C\neq\emptyset$. Since $v\in C$ and group identities 
are unique, we get that $C\subset\gud{v}{v}$.

We first show for all $a\in C^*(G)$ and `appropriate' $g,f\in C_c(G)$, 
$$ (l^{u_n}(a)g\mid h)_{u_n} \to (l^C(a)g\mid h)_C $$
weakly. We will shortly explain what we mean by `appropriate', and then show that we can find such $g$ and $h$.

From \cite[Corollary 1.4]{Ren91}
there is a continuous choice of Haar measures on $\s$. Specifically, the 
map from $\s$ into $\mathbb{C}$ given by  
\begin{equation} \label{contChoiceHaar2}
H\mapsto \int_{H}{f(t)\Delta_{H}(t)\:d\beta^{H}(t)} 
\end{equation} 
is continuous for every $f\in C_{c}(G)$ (see also the proof of Lemma 3.1 in \cite{Wyk18} for details on the inclusion of the modular function in the integrand). 
Let $w\in\gus$ and let $H$ be a closed subgroup of $\gud{w}{w}$. Let  $f\in C_{c}(G)$ and $g,h\in C_{c}(G_w)$.
Recall that the representation  $l^H$  acts on the completion of $C_{c}(G_w)$
with respect to the inner product
$$(l^H(f)g\mid h)_H:=\int_{H}{(h^{*}\ast f\ast g)(t)\Delta_{H}(t)^{-1/2}\:d\beta^{H}(t)}. $$
We claim that since 
$\gud{u_n}{u_n}\to C$, it follows from the continuity of (\ref{contChoiceHaar2}) that 
\begin{equation}\label{innerconvCCR}
(l^{u_n}(f)g\mid h)_{u_n}\to (l^C(f)g\mid h)_C.
\end{equation} 
Now, by `appropriate' $g$ and $h$ above we mean that $g$ an $h$ have to be defined on $G_{u_n}$ for all $n$ and on $C\subset \gud{v}{v}$. Since the 
source map is continuous all the sets $G_w$ $(w\in\gus)$ are closed. Thus
they are also locally compact in the relative topology from $G$. 
Then by Lemma 1.42 in \cite{CrossedProd} we can extend any continuous compactly
supported function on $G_w$ to $C_{c}(G)$. Thus every $g\in C_c(G_w)$ is the restriction of some $\tilde{g}\in C_c(G)$ and it suffices to take $f,g,h\in C_{c}(G)$. Then, since $\gud{u_n}{u_n}\to C$, we see that 
(\ref{innerconvCCR}) follows from the continuity of (\ref{contChoiceHaar2}).

Now let $\epsilon>0$, $g,h\in C_{c}(G)$ and $a\in C^*(G)$. Then 
for any fixed $n$ there is an
$f\in C_{c}(G)$ such that 
\begin{equation} \label{anotherApproxCCR}
||a-f||_{C^*(G)}< \min\left\lbrace\frac{\epsilon}{3||g||_{u_n}||h||_{u_n}},
\frac{\epsilon}{3||g||_{C}||h||_{C}}\right\rbrace.
\end{equation}
For $n$ large enough,
(\ref{innerconvCCR}) and the Cauchy-Schwarz inequality 
together with (\ref{anotherApproxCCR}) imply that 
\begin{eqnarray*}
	&&\mid (l^{u_n}(a)g\mid h)_{u_n} - (l^C(a)g\mid h)_C \mid \\
	&\leq & \mid(l^{u_n}(a)g\mid h)_{u_n}-(l^{u_n}(f)g\mid h)_{u_n}\mid + 
	\mid(l^{u_n}(f)g\mid h)_{u_n} - (l^C(f)g\mid h)_C\mid  +  \\
	&& \hspace{0.5cm}
	\mid(l^C(f)g\mid h)_C - (l^C(a)g\mid h)_C\mid  \\
	&\leq & ||l^{u_n}(a-f)||\:||g||_{u_n}\:||h||_{u_n} + 
	\frac{\epsilon}{3}  + 
	||l^C(f-a)|| \:||g||_{C}\:|| h||_{C} \\
	&\leq & ||a-f||_{C^*(G)}\:||g||_{u_n}\:||h||_{u_n} + 
	\frac{\epsilon}{3}  + 
	||a-f||_{C^*(G)} \:||g||_{C}\:|| h||_{C} \hspace{1cm} \\
	&<& \epsilon.
\end{eqnarray*} 
Hence we have now verified (\ref{innerconvCCR}).

Next we show that $\ker(l^{u_n})\subset\ker(l^C)$ for every $n$.  By Lemma 5.1 in \cite{Cla07} the map from $\gus/G$ into the spectrum
$C^*(G)^{\wedge}$  given by $[u]\mapsto [l^u]$ is well-defined. 
Since $u_n\in[u]$ for all $n$, we have 
$l^{u_m}\sim l^{u_n}$ for all $m$ and $n$. Thus  $\ker(l^{u_m})=\ker(l^{u_n})$
for all $m$ and $n$. Let $a\in\ker(l^{u_n})$ and let $g,h\in C_{c}(G)$. Then $(l^{u_n}(a)g\mid h)_{u_n}= 0$. Since 
$(l^{u_n}(a)g\mid h)_{u_n} \to (l^C(a)g\mid h)_C$ for every $g,h\in C_{c}(G)$, it follows that $(l^C(a)g\mid h)_C = 0$ for every $f,g,h\in C_{c}(G)$. Hence $l^C(a)=0$ and thus $a\in\ker(l^C)$, completing the proof.
\end{proof}

Before we prove our main result (Theorem \ref{mainCCRthm} below)
we need a last lemma that will be vital. 
This is Lemma \ref{lem:EqReps} below and is a vector-valued version of Lemma 4.15 from \cite{Wil81}. Lemmas \ref{lem:EqRepsPrequal} and \ref{lem:EqReps} and their proofs were shown to us by Dana Williams. We begin with the set-up needed for these two lemmas. 

Suppose that $Y$ is a second-countable locally compact Hausdorff space,  $\mu$ is Radon measure on $Y$ and $\fH{}$ a separable Hilbert space. For $y\in Y$, $h\in L^2(Y,\fH{},\mu)$ and each $\psi\in L^{\infty}(Y)$ we obtain a multiplication operator $L_{\psi}$ in $B(L^2(Y,\fH{},\mu))$ 
defined  by 
\begin{equation*}
L_{\psi}(h)(y):=\psi(y)h(y).
\end{equation*}
Such operators are called diagonal operators. 

Suppose that $X$, $Y$
and $Z$ are second-countable locally compact Hausdorff spaces. Let $\mu_Y$ and $\mu_Z$ be Radon measures 
on $Y$ and $Z$, respectively. So $\mu_Y$ and $\mu_Z$ correspond to  
linear functionals on $C_{c}(Y)$ and $C_{c}(Z)$, respectively, via the Riesz 
representation theorem \cite[Theorem 2.14]{RudRC}. Let 
$i:Y\to X$ and $j:Z\to X$  be continuous injections and let $\fH{Y}$ and $\fH{Z}$ be separable Hilbert spaces.
Let $M_Y:C_{0}(X) \to B(L^2(Y,\fH{Y},\mu_Y))$ be given by 
$$(M_Y(\phi)f)(y):=\phi(i(y))f(y). $$ Similarly let 
$M_Z:C_{0}(X) \to B(L^2(Z,\fH{Z},\mu_Z))$ be given by 
$$(M_Z(\phi)g)(z):=\phi(j(z))g(z). $$

\begin{lem} \label{lem:EqRepsPrequal}
	Let $\mathcal{L}$ denote the algebra of all diagonal operators in $B(L^2(Y,\fH{Y},\mu_Y))$ and $\mathcal{L}'$ its commutant. If $T\in M_Y(C_0(X))'$ then $T\in \mathcal{L}'$. 
\end{lem}

\begin{proof}
	Let $\mathcal{B}^b(Y)$ denote the bounded Borel functions on $Y$ and suppose $\psi \in \mathcal{B}^b(Y)$ and $T\in M_Y(C_0(X))'$. We have to prove that $(L_{\psi}T)h=(TL_{\psi})h$ for every $h\in L^2(Y,\fH{Y},\mu_Y)$. By continuity  it will suffices to find a sequence $\{\phi_n\}\subset C_0(X)$ such that $M_Y(\phi_n)\to L_{\psi}$ in the strong operator topology.  
	
	We may assume $\psi(y) \leq 1$ for every $y\in Y$. Since $Y$ is second-countable and locally compact we can write $Y=\bigcup_n K_n$, where each $K_n$ is a compact subset of $Y$ with $K_n$ contained in the interior of $K_{n+1}$. Since $i:Y\to X$ is continuous it follows that $i(K_n)$ is compact in $X$ for every $n$. Hence there is a $\phi_n\in C_c(X)$ such that $0\leq | \phi_n(x)|\leq 1$ for every $x\in X$ and $\phi_n(i(y))=\psi(y)$ for every $y\in K_n$ (see for example Lemmas 1.41 and 1.42 in \cite{CrossedProd}).  
	
	Fix $y\in Y$, $h\in L^2(Y,\fH{Y},\mu_Y)$ and let $1_{K_n}$ denote the characteristic function on $K_n$. Since  $Y=\bigcup_n K_n$ and $K_n\subset K_{n+1}$ there is a $n_0\in \mathbb{N}$ such that $y\in K_n$ for every $n>n_0$. Then for all $n>n_0$ we have  
	$$|| \phi_n(i(y))h(y)-\psi(y)h(y)  ||=0. $$
	Hence $(\phi_n\circ i) h\to \psi h$ pointwise.
	Since $|| \phi_n(i(y))h(y) ||\leq ||h(y)||$ for all $y\in Y$ and $h\in L^2(Y,\fH{Y},\mu_Y)$, it follows from the Dominated Convergence Theorem that 
	$\psi h \in L^2(Y,\fH{Y},\mu_Y)$ and 
	\begin{eqnarray*}
		0&=& \lim_{n}\int_{Y}{|| \phi_n(i(y))h(y)-\psi(y)h(y)  ||^2 \: d\mu_Y(y)} \\
		 &=&  \lim_{n}||M_Y(\phi_n)h\to L_{\psi}h||^2.
	\end{eqnarray*}
	 Hence  $M_Y(\phi_n)\to L_{\psi}$ strongly, which is sufficient.
\end{proof}

\begin{lem} \label{lem:EqReps}
Let $X,Y,Z,\mu_Y,\mu_Z, i$, $j$, $\fH{Y}$, $\fH{Z}$ $M_Y$ and $M_Z$ be as above.
If \newline 
$i(Y)\cap j(Z)=\emptyset$, then 
$M_Y$ and $M_Z$ have no  unitarily equivalent sub-representations. 
\end{lem}
\begin{proof}
Suppose $\mathcal{E}$ is a non-zero invariant subspace of $L^2(Y,\fH{Y},\mu_Y)$ for $M_Y$, and $P_{\mathcal{E}}$ is the orthogonal projection onto $\mathcal{E}$. Then $P_{\mathcal{E}}\in M_Y(C_0(X))'$. Hence by Lemma \ref{lem:EqRepsPrequal} we have $P_{\mathcal{E}}\in \mathcal{L}'$ and thus $P_{\mathcal{E}}$ is decomposible by \cite[Theorem F.21]{CrossedProd}. Hence $P_{\mathcal{E}}$ can be written as a direct integral of Borel operators $P:y\mapsto P(y), y\in Y$ (see for example Appendix F in \cite{CrossedProd} or \cite{DixVN} for definitions and more on decomposible operators and direct integrals). Since $P$ is determined $\mu_Y$-a.e. and $P=P^2=P^*$ we can modify $P$ on null sets and assume that $P(y)=P(y)^2=P(y)^*$ for every $y\in Y$. The set $E:=\{y\in Y:||P(y)||=1\}$ must have non-zero measure. Hence for any non-null Borel set $B\subset E$, the direct integral 
\begin{equation*}
P_{\mathcal{E}}^B = \int_Y^{\oplus}{1_B(y)P(y)\: d\mu_Y(y) }
\end{equation*}
is a non-zero projection in $\mathcal{L}'$. Therefore $P_{\mathcal{E}}^B$ determines a sub-representation  $M_Y^B$ of $M_Y^E$ determined by $P_{\mathcal{E}}$. The point is that every sub-representation of $M_Y$ has the form $M_Y^B$ for some non-null Borel set $B$.
The same conclusions apply to any non-zero invariant subspace $\mathcal{F}$  of $L^2(Z,\fH{Z},\mu_Z)$ for $M_Z$. 

Suppose that a sub-representation $M_Y^E$ of $M_Y$ is equivalent to a sub-representation $M_Z^F$ of $M_Z$. 
Since $\mu_Y$ is a Radon measure there is a compact set $C\subset E$ with $\mu_Y(C)$ non-zero. Since $C$ is a Borel set, $M_Y^C$ is a sub-representation of $M_Y^E$ and is equivalent to a sub-representation $M_Z^A$ of $M_Z^F$. Also, since $\mu_Z$ is a Radon measure there is a compact set $D\subset A$ with $\mu_Z(D)$ non-zero and $M_Z^D$ equivalent to a sub-representation of $M_Y^C$. By Urysohn's lemma (see for example Lemma 4.41 in \cite{CrossedProd}) and since $i(Y)\cap j(Z)=\emptyset$, we can find a $f\in C_0(X)$ such that $f(c)=1$ for all $c\in C$ and $f(d)=0$ for all $d\in D$. Then $M_Z^C$ is a non-zero operator but $M_Z^D$ is the zero operator. This contradicts the equivalence of  $M_Z^D$ to a sub-representation of $M_Y^C$, which completes the proof.
\end{proof}

We now prove our main CCR result.  

\begin{thm}  \label{mainCCRthm}
	Let $G$ be a second-countable locally compact Hausdorff groupoid with a Haar system. If	$C^{*}(G)$ is CCR then $\gus/G$ is $T_{1}$ and the stabilizers of $G$ are CCR. 	
\end{thm}
 	
\begin{proof}
	Assume $C^*(G)$ is a CCR $C^*$-algebra. We first show that the orbit space is $T_1$. To do this it will suffice to show that orbits are closed in 
  	$\gus$. 
  	We prove the orbits are closed by contradiction.
  	Suppose that $C^{*}(G)$ is CCR and that there is a $u\in \gus$ such that 
 	$[u]$ is not closed in $\gus/G$. Let $v\in\overline{[u]}\backslash [u]$. Then there 
 	is a sequence $\{u_{n}\}\subset [u]$ such that $u_{n}\to v$. We will show that
 	the trivial representation of $\gudu$ induced up to $C^*(G)$ is equivalent to 
 	an irreducible representation of $\gud{v}{v}$ induced up to $C^*(G)$. By using Lemma \ref{lem:EqReps} this equivalence then leads to a contradiction, so that $[u]$ must be closed.
 	
 	Since $\{u_{n}\}\subset [u]$ and $u_{n}\to v$ in $\gus$,
 	we have by Lemma \ref{convInducedTrivialReps}
 	that there exists a closed subgroup $C\subset\gud{v}{v}$ such that 
 	$\ker(l^{u_{n}})\subset \ker(l^C)$ for all $n$.
 	By \cite[Lemma 5.1]{Cla07} the map $[u]\mapsto [l^u]$ from $\gus$
 	into the spectrum $C^{*}(G)^{\wedge}$ of $C^*(G)$ is well defined. 
 	Also, since $\{u_{n}\}\subset [u]$, we have 
 	$l^{u_{n}}\sim l^u$ for every $n$.
 	Thus we have 
 	\begin{equation} \label{CCRmainContainement1}
 		\ker(l^{u}) \subset \ker(l^C).
 	\end{equation}
 	
 	We claim that there is a $\pi\in C^*(\gud{v}{v})^{\wedge}$ such that 
 	\begin{equation}\label{CCRmainContainement2}
 	\ker(\ind{C}{\gud{v}{v}}{1}) \subset \ker(\pi).
 	\end{equation}
 	To see this, note that if $\ind{C}{\gud{v}{v}}{1}$ is irreducible then
 	the claim follows immediately by taking
 	$\pi=\ind{C}{\gud{v}{v}}{1}$. If $\ind{C}{\gud{v}{v}}{1}$ is not irreducible
 	then $  \ker(\ind{C}{\gud{v}{v}}{1})$ is a closed two-sided ideal in 
 	$C^*(\gud{v}{v})$. By Proposition A.17 in \cite{MoritaEq} every 
 	ideal in a $C^*$-algebra is equal to the intersection of all 
 	the primitive ideals containing it. Thus 
 	$$\ker(\ind{C}{\gud{v}{v}}{1})=\cap\{P\in\Prim(C^*(\gud{v}{v})): 
 		\ker(\ind{C}{\gud{v}{v}}{1})\subset P\},  $$
 	and so there is a $ P\in\Prim(C^*(\gud{v}{v}))$ such that 
 	$\ker(\ind{C}{\gud{v}{v}}{1})\subset P$. Since $P$ is a primitive 
 	ideal, it follows that $P=\ker(\pi)$ with $\pi\in C^*(\gud{v}{v})^{\wedge}$,
 	proving the claim. 
 	
 	By \cite[Theorem 4]{IonWil09} we may induce representations in stages. That is, 
 	since we have the inclusion of (sub)groupoids $C\subset\gud{v}{v}\subset G$, 
 	the representations $ \ind{\gud{v}{v}}{G}{(\ind{C}{\gud{v}{v}}{1})}$ and
 	$l^C=\ind{C}{G}{1}$ are equivalent. Then 
 	\begin{equation}\label{CCRmainContainement3}
 	\ker(l^C)=\ker(\ind{\gud{v}{v}}{G}{(\ind{C}{\gud{v}{v}}{1}})).
 	\end{equation}
 	
 	Combining the inclusions from (\ref{CCRmainContainement1}),
 	(\ref{CCRmainContainement3}) and (\ref{CCRmainContainement2}) we have
 	\begin{equation}
 	\ker(l^{u}) \subset \ker(l^C)
 				=\ker(\ind{\gud{v}{v}}{G}{(\ind{C}{\gud{v}{v}}{1})})
 				\subset \ker(\ind{\gud{v}{v}}{G}{\pi}).
 	\end{equation}
 	Now, since $l^{u}$ and  $\ind{\gud{v}{v}}{G}{\pi}$ are irreducible representations
 	of $C^*(G)$ (\cite[Theorem 5]{IonWil09}) and  $\ker(l^{u})\subset  \ker(\ind{\gud{v}{v}}{G}{\pi})$, it 
 	follows from Lemma \ref{ccrrepcontainiseqv} that 
 	$l^{u}\sim\ind{\gud{v}{v}}{G}{\pi}$. Let $L_u$ and $L_v$ denote
 	the extensions of $l^u$ and $\ind{\gud{v}{v}}{G}{\pi}$,
 	respectively, to the multiplier algebra $\mathcal{M}(C^*(G))$ of $C^*(G)$.
 	 Then $l^{u}\sim\ind{\gud{v}{v}}{G}{\pi}$ implies that 
 	$L_u\sim L_v$. 
  	On the other hand, by Proposition
 	\ref{mainEquivReps},  we have that $L_u\circ V$ 
 	is unitarily equivalent to the representation $M_u$ of $C_{0}(\gus)$ 
 	by multiplication operators on $L^2(G_u /\gudu,\fH{u},\sigma_u)$. Similarly, for the extension $L_v$ 
 	of $\ind{\gud{v}{v}}{G}{\pi}$ to $\mathcal{M}(C^*(G))$ we have that 
 	$L_v\circ V$ 
 	is unitarily equivalent to a representation $M_v$ of $C_{0}(\gus)$ 
 	as multiplication operators on $L^2(G_v /\gud{v}{v},\fH{v},\sigma_v)$. Then 
 	it follows that 
 	$$ M_u\sim L_u\circ V\sim L_v\circ V\sim M_v.$$
 	However, since $[u]\cap [v]=\emptyset$ 
 	the equivalence $ M_u\sim  M_v$ contradicts Lemma \ref{lem:EqReps}.
 	Hence $[u]$ is closed in $\gus/G$, showing that the 
 	orbit space $\gus/G$ is $T_1$.
 	
 	Lastly, we need to show that the stabilizers of $G$ are CCR. However, since the orbits of $G$ are closed by the proof above, exactly the same argument as in \cite[Theorem 6.1]{Cla07} shows that the stabilizers of $G$ are CCR.
 \end{proof}

We can combine Theorem \ref{mainCCRthm} and Clark's Theorem 6.1 in \cite{Cla07} to formulate an improved characterization of CCR groupoids $C^*$-algebras without amenability:

\begin{thm} \label{mainCCRthm2}
		Let $G$ be a second-countable, locally compact and Hausdorff groupoid
		with a Haar system. Then $C^*(G)$ is 
		CCR if and only if the stabilizers of $G$ are CCR 
		and $\gus/G$ is $T_{1}$.
\end{thm}

\section{EXAMPLES} \label{chap:Examples}
We give two simple examples of groupoids that illustrate the 
improved GCR (\cite[Theorem 5.3]{Wyk18}) and CCR characterizations of groupoid $C^*$-algebras. The first example is a groupoid with non-amenable stabilizers, which is GCR but not CCR. The second example is a group-bundle groupoid 
with non-amenable stabilizers which is CCR.

\begin{itemize}[leftmargin=*]

\item[{(1)}] \label{ex:transfGroupExGCR}

Fix any $n\geq 2$. Consider the group $\GL_n(\mathbb{R})$ of all $n\times n$ invertible matrices with real entries.  The map 
from $\GL_n(\mathbb{R})$ to $\mathbb{R}^{n^2}$ given by 
$$(a_{ij})\mapsto (a_{11},a_{12},\ldots, a_{nn})$$
%
is a bijection. Via this bijection we give  $\GL_n(\mathbb{R})$ the topology obtained from the product topology on $\mathbb{R}^{n^2}$.
Then $\GL_n(\mathbb{R})$ is a second-countable, locally 
compact and Hausdorff group. Let $A\in\GL_n(\mathbb{R})$ 
and $x\in \mathbb{R}^{n}$. Then 
$$ A\cdot x := \frac{1}{\det(A)}\:x$$
defines an action of 
$\GL_n(\mathbb{R})$ on $\mathbb{R}^{n}$. 
Since the $\det$ is a continuous function and 
scalar multiplication is continuous in $\mathbb{R}^{n}$, it follows that this
action is continuous. Hence we have a second-countable, 
locally compact and Hausdorff transformation
group $(\GL_n(\mathbb{R}),\mathbb{R}^{n})$.
Then $(\GL_n(\mathbb{R}),\mathbb{R}^{n})$ can be identified with a groupoid $G:= \GL_n(\mathbb{R})\times \mathbb{R}^n$ such that the orbits, orbit space and stabilizers correspond to their transformation group analogues
(see Example 3.3 in \cite{HueCla08} for the details).
Suppose that $\beta$ is a Haar measure on $\GL_n(\mathbb{R})$ and for
every $x\in\mathbb{R}^{n}$ let $\delta_{x}$ denote the point mass measure. Then 
$\{\beta\times \delta_{x}\}_{x\in\mathbb{R}^{n}}$  is a Haar system 
for $G$.

Consider the stabilizers. Let $x\in \mathbb{R}^n\backslash\{0\}$
and let $\SL_n(\mathbb{R})$ be the subgroup of  $\GL_n(\mathbb{R})$
of all $n\times n$ matrices with determinant $1$ (i.e. the special linear group).
Then the stabilizer at $x$ is 
$$ \gud{x}{x}= \{A\in \GL_3(\mathbb{R}): A\cdot x=x\}
=\{A\in \GL_3(\mathbb{R}): \det(A)=1 \}= \SL_3(\mathbb{R}).$$
The stabilizer at $0$ is 
$$ \gud{0}{0}= \{A\in \GL_3(\mathbb{R}): A\cdot0=0\}
= \GL_3(\mathbb{R}).$$  
Note that all the stabilizers are non-amenable 
(see for example Ex 1.2.6 on p.28 in \cite{RundeLecAmen}), and they do 
not vary continuously (there is a discontinuity at $0$). 
The stabilizers are CCR (\cite{Har54}), 
and thus also GCR. 

We show that the orbit space is $T_0$, but not $T_1$.
Suppose that $x\in \mathbb{R}^n$. The orbit 
of $x$ is given by 
$$[x]:=\{A\cdot x: A\in \GL_n(\mathbb{R})\} = 
\{cx: c\in\mathbb{R}, c\neq0  \}. $$
With the exception of $0$, each orbit is a straight line in 
$\mathbb{R}^n$ through the 
origin, but with the origin excluded (since $\det(A)\neq0$). 
Thus these orbits are not closed as subset of $\mathbb{R}^n$. That is, 
the orbit space cannot be $T_1$. Since the orbit space is not $T_1$, 
Theorem \ref{mainCCRthm}
implies that $C^{*}(G)$ is not CCR (Note: since the 
stabilizers are not amenable, we cannot draw this conclusion 
from either \cite[Theorem 6.1]{Cla07} nor \cite[Proposition 4.17]{Wil81}.)

Since the orbits are open in their closures, it follows that 
the orbits are locally closed (\cite[Lemma 1.25]{CrossedProd}). 
Then the equivalence relation
$R(\gamma):=(r(\gamma),s(\gamma)),\gamma\in G,$ on the unit space is an $F_{\sigma}$ set (\cite[Lemma 5.1]{Wyk18}), and it follows from 
\cite[Theorem 2.1]{Ram90} that the orbit space is $T_0$. Thus, because
the stabilizers of $G$ are GCR and the orbit space is 
$T_0$, it follows from Theorem 5.3 in \cite{Wyk18} that $C^*(G)$ is GCR. 
\\

\item[{(2)}] \label{ex:CCRgroupBundle}
The next example is a CCR group-bundle groupoid. Suppose that $\{G_{i}\}_{i\in\mathbb{I}}$ is a family of groups indexed by a countable set
$\mathbb{I}$. Let $G:=\{(i,t):i\in\mathbb{I}, t\in G_i\}$. Then $G$ is a groupoid, where $(i,s),(j,t)\in G$ are composable if and only if $i=j$, and then $ (i,s)(i,t):=(i,st)$ and $(i,t)^{-1}:=(i,t^{-1})$. Then $G$ is called a group-bundle groupoid. 
Suppose that $G$
is second-countable, locally compact and Hausdorff with a Haar system. 
Let  $e_i$ denote the identity of the group $G_i$. 
Then $r(i,t)=s(i,t)=(i,e_i)$. Also, the unit space $\gus$ is 
homeomorphic to the orbit space $\gus/G$ and since the unit space is Hausdorff the 
orbit space is also Hausdorff. Hence the orbit space is always $T_{0}$ and $T_{1}$.
Thus we have following corollary to Theorem \ref{mainCCRthm2}:

\begin{cor}\label{bundleResult}
	Suppose that $G=\sqcup_{i\in\mathbb{I}}G_i$ is a second-countable 
	locally compact Hausdorff group-bundle 
	groupoid with a Haar system. Then $C^*(G)$ is CCR  if and only if 
	$C^*(G_{i})$ is CCR for every $i\in\mathbb{I}$.  
\end{cor}

By Theorem 5.3 in \cite{Wyk18} the same statement holds for GCR groupoid $C^*$-algebras.

	For a specific example, fix $n\geq 2$ and take $G=\sqcup_{i}\:\textrm{SL}_n(\mathbb{R})$ with 
	the disjoint union topology (see Example 2.1.5 in \cite{Wyk17} for details on the disjoint union topology). Since the 
	stabilizers are all the same, it follows that the 
	stabilizer map is continuous. Thus by \cite[Lemmas 1.1 and 1.3 ]{Ren91}, $G$ has a Haar system $\{\lambda^n\}$, where each $\lambda^n$ is a	Haar measure of $SL_n(\mathbb{R})$ (see Example 2.1.5 in \cite{Wyk17} that the Haar system is indeed given by the Haar measure on $\textrm{SL}_n(\mathbb{R})$). 
	Every stabilizer $SL_n(\mathbb{R})$ is CCR 
	(\cite{Har54}). Thus $C^*(G)$ is CCR by Corollary \ref{bundleResult}. \\
\end{itemize}

\section*{ACKNOWLEDGEMENTS}
This paper is a product of work done in the author's doctoral thesis. I thank my advisors Astrid an Huef and Lisa Orloff Clark for their guidance, support and endless patience   throughout my studies. In particular, I thank Astrid for a carefully reading this paper and for valuable comments. I thank Dana Williams for bringing Lemmas \ref{lem:EqRepsPrequal} and \ref{lem:EqReps} to my attention.

\bibliographystyle{acm}

\end{document}